\documentclass{amsart}
\usepackage{cases}
\usepackage{amsmath}
\usepackage{amsfonts}
\usepackage{pifont}
\usepackage{graphicx,amssymb,mathrsfs,amsmath}
 \newtheorem{thm}{Theorem}[section]
 \newtheorem{cor}[thm]{Corollary}
 \newtheorem{lem}[thm]{Lemma}
 \newtheorem{prop}[thm]{Proposition}
 \theoremstyle{definition}
 \newtheorem{defn}[thm]{Definition}
 \newtheorem{example}[thm]{Example}
 
 \theoremstyle{remark}
 \newtheorem{rem}[thm]{Remark}
 \newtheorem{rems}[thm]{Remarks}
 \numberwithin{equation}{section}
 \newcommand{\eps}{\varepsilon}

 \newcommand{\Real}{\mathbb{R}}
 \newcommand{\Complex}{\mathbb{C}}

\begin{document}

\title[On fractional powers of generators]{On fractional powers of generators \\of fractional resolvent families}
\thanks{This project is supported by the NSF of China (No.10971146).}
\date{}
\author{Miao Li}
\address{Department of Mathematics, Sichuan University, Chengdu 610064, P.R.China} \email{mli@scu.edu.cn}
\author{Chuang Chen}
\address{Department of Mathematics, Sichuan University, Chengdu 610064, P.R.China} \email{cc0508036@163.com}
\author{Fu-Bo Li}
\address{Department of Mathematics, Sichuan University, Chengdu 610064, P.R.China} \email{lifubo@scu.edu.cn}
\maketitle

\begin{abstract}

%We show that if $-A$ generates a bounded $\alpha$-times resolvent
%family for some $\alpha \in (1,2]$, then $-A^{1/\alpha}$ generates
%an analytic $C_0$-semigroup. Moreover, if $-A$ generates a bounded
%$\alpha$-times resolvent family for some $\alpha \in (0,2)$, then
%$-A^{\beta}$ generates an analytic $\alpha$-times resolvent family
%for $\beta \in (0,1)$. And more generation theorems are given for
%the fractional powers of generators of fractional resolvent
%families. Such relations are applied to study the solutions of
%Cauchy problems of fractional order and first order.

We show that if $-A$ generates a bounded $\alpha$-times resolvent
family for some $\alpha \in (0,2]$, then $-A^{\beta}$ generates an
analytic $\gamma$-times resolvent family for $\beta
\in(0,\frac{2\pi-\pi\gamma}{2\pi-\pi\alpha})$ and $\gamma \in
(0,2)$. And a generalized subordination principle is derived. In
particular, if $-A$ generates a bounded $\alpha$-times resolvent
family for some $\alpha \in (1,2]$, then $-A^{1/\alpha}$ generates
an analytic $C_0$-semigroup. Such relations are applied to study the
solutions of Cauchy problems of fractional order and first order.

\end{abstract}

%%%%%%%%%%%%%%%%%%%%%%%%%%%%%%%%%%%%%%%%%%%%%%%%%%%%%%%%%%%%%%%%%%%%%

\section{Introduction}\label{1}
Let $A$ be a closed densely defined linear operator on a Banach
space $X$.  The resolvent families were introduced by Da Prato
\cite{Da} to study Volterra integral equations of the form
\begin{equation}\label{0.1}
u(t) = f(t) + A\int_0^t a(t-s)u(s) ds.
\end{equation}
A family $\{R(t)\}_{t \ge 0} \subset B(X)$ is called a resolvent
family for $A$ with kernel $a$ if

(a) $R(0)=I$ and $R(t)$ is strongly continuous;

(b) $AR(t) \subseteq R(t)A$ for every $t \ge 0$;

(c) for every $x \in D(A)$,
$$
R(t)x = x + \int_0^t a(t-s)R(s)Axds.
$$
It is shown that the problem (\ref{0.1}) is well-posed (in the sense
of \cite{Pru}) if and only if there is a resolvent family for $A$.
Since a $C_0$-semigroup is a resolvent family for its generator with
kernel $a_1(t) \equiv 1$, and a cosine operator function is a
resolvent family for its generator with kernel $a_2(t) = t$, it is
natural to consider the resolvent family with kernel $a_\alpha(t) =
\frac{t^{\alpha-1}}{\Gamma(\alpha)}$. Also note the following facts:
if $A$ generates a $C_0$-semigroup, then the Cauchy problem of first
order
$$
u'(t)= Au(t), \,t \ge 0; \;u(0)=x
$$
is well-posed; and if $A$ generates a cosine operator function, then
the second order Cauchy problem
$$
u''(t)=Au(t),\,t \ge 0; \;u(0)=x, \, u'(0)=y,
$$
is also well-posed. This motivates one to consider the relations
between the existence of resolvent family for $A$ with kernel
$a_\alpha(t) $ and the well-posedness of some kind of fractional
Cauchy problem $ D_t^\alpha u(t)=Au(t) $ with proper initial values.
Such relation was proved by Bajlekova \cite{Baj} in 2001. The
resolvent family for $A$ with kernel $a_\alpha$ was therefore called
$\alpha$-times resolvent family. For more general resolvent families
see \cite{Liz1, Liz2}.

On the other hand, it is well known that if $-A$ generates a bounded
cosine function operator, then $-A^{1/2}$ generates an analytic
$C_0$-semigroup of angle $\pi/2$ (cf. \cite{Ke}). And it was proved
by Yosida in 1960 (cf. \cite{Ka, Yo}) that if  $T$ is a bounded
$C_0$-semigroup on a complex Banach space $X$, with the generator
$A$, then $-A^\alpha,0<\alpha<1$, generates an analytic semigroup
$T_\alpha$ on $X$, and $T_\alpha$ is subordinated to $T$ through the
L\'{e}vy stable density function.

For $\alpha$-times resolvent family, the questions of interest are:

 $(Q_1)$ If $-A$ generates a bounded $C_0$-semigroup, does
 $-A^\alpha$ generate an $\alpha$-times resolvent family?

 $(Q_2)$ If $-A$ generates a bounded $\alpha$-times resolvent family, does
 $-A^{1/\alpha}$ generate a $C_0$-semigroup?

 $(Q_3)$  If $-A$ generates a bounded $\alpha$-times resolvent family,
 does $-A^{1/2}$ generate an $\alpha/2$-times resolvent family?

 $(Q_4)$ If $-A$ generates a bounded $\alpha$-times resolvent family,
 does $-A^\beta$ also generate an $\alpha$-times resolvent family
 for some suitable $\beta$?

 $(Q_5)$ If $-A$ generates a bounded $\alpha$-times resolvent family,
 does $-A^\beta$ generate a $\gamma$-times resolvent family
 for some suitable $\beta$ and $\gamma$?

Our first aim in this paper is trying to give answers to the above
questions in a unified way. We first note the fact: if $-A$ is the
generator of a bounded $\alpha$-times resolvent family, then $A$ is
a sectorial operator (see Section 2 for details). Therefore, it is
possible to define the fractional power $A^b$ for $b > 0$. By using
the theory of functional calculus for sectorial operators (see
\cite{Bal, Ha, Ko, MS}), we are able to give positive answers to the
questions above. These relations are clarified in Section \ref{3}.

The second purpose of this paper is to establish connections between
solutions of fractional Cauchy problems and Cauchy problems of first
order. Obverse that many phenomena in the theory of stochastic
processes, finance and hydrology are recently described through
fractional evolution equations, see \cite{BMN, BWM2, SGM, Za} and
the references therein. For example, Zaslavsky \cite{Za} introduced
the fractional kinetic equation
\begin{eqnarray}\label{1.3}
\begin{split}
&D_t^\alpha u(t,x)+ L_xu(t,x)=0,\quad t  >0\\
&u(0,x)=f(x),
\end{split}
\end{eqnarray}
for Hamiltonian chaos, where $\alpha\in(0,1)$,  $-L_x$ is the
generator of some continuous Markov process, and $D_t^\alpha$ is
understood the Caputo fractional derivative in time (see Section 2).
Baeumer and Meerschaert \cite{BM}, and Meerschaert and Scheffler
\cite{MS2} showed that the fractional Cauchy problem (\ref{1.3}) is
related to a certain class of subordinated stochastic processes.
More precisely, Theorem 3.1 in \cite{BM} shows that the formula
\begin{eqnarray}\label{1.5}
\begin{split}
u(t,x)=\int_0^\infty v((t/s)^\alpha,x)b_\alpha(s)ds,
\end{split}
\end{eqnarray}
yields a unique strong solution of (\ref{1.3}), where $b_\alpha$ is
the smooth density of the stable subordinator such that the Laplace
transform $\widehat{b_\alpha}(\lambda)=\int_0^\infty e^{-\lambda
t}b_\alpha(t)dt=e^{-\lambda^\alpha}$ and $v$ is the solution of
\begin{eqnarray}\label{1.4}
\begin{split}
&v_t'(t,x)+L_xv(t,x)=0,\quad t >0\\
&v(0,x)=f(x).
\end{split}
\end{eqnarray}
The formula (\ref{1.4}) can also be explained by the subordination
principle for fractional resolvent family, see Theorem 3.1 in
\cite{Baj}or Lemma \ref{subordinate}. If the fractional power of
$L_x$, $L_x^\alpha$, is defined, it is also of interest to know the
relations between the solution of (\ref{1.4}) and that of
\begin{eqnarray}\label{1.31}
\begin{split}
&D_t^\alpha u(t,x)+L_x^\alpha u(t,x)=0,\quad t >0\\
&u(0,x)=f(x).
\end{split}
\end{eqnarray}
In Section 4 we will give this connection.
%Select
%$x\in\mathbb{R}^d$ and let $X(t)=x+X_0(t)$. By Theorem 5.1 in
%\cite{MS2}, for $T(t)f(x)=E_x[f(X(t))]$, the semigroup on
%$L_1(\mathbb{R}^d)$ associated with the operator L\'{e}vy motion
%$X(t)$, it follows that (\ref{1.5}) also equals
%$u(t,x)=E_x[f(Z_t)]$, where $Z_t=X(E_t)$ is the CTRW scaling limit
%process. Therefore, the subordinated process $Z_t$ is the stochastic
%solution of (\ref{1.3}).
Moreover, Baeumer, Meerschaert and Nane \cite{BMN} proved that Eq.
(\ref{1.3}) with $\alpha = 1/2$ and the initial value problem
\begin{eqnarray}\label{1.6}
\begin{split}
&u_t'(t,x)-L_x^2u(t,x)+\frac{t^{-1/2}}{\Gamma(1/2)}L_xf(x)=0,\quad t >0\\
&u(0,x)=f(x),
\end{split}
\end{eqnarray}
have the same solution; and (\ref{1.3}) with $\alpha = 1/3$ and
\begin{eqnarray}\label{1.7}
\begin{split}
&u_t'(t,x)+L_x^3u(t,x)+\frac{t^{-2/3}}{\Gamma(1/3)}L_xf(x)-\frac{t^{-1/3}}{\Gamma(2/3)}L_x^2f(x)=0,\quad t >0\\
&u(0,x)=f(x),
\end{split}
\end{eqnarray}
have the same solution, respectively. Another example is given by
Allouba and Zheng \cite{AZ} and DeBlassie \cite{De}, they consider
the case that $L_x = -\Delta$, the Laplace operator.
%They showed
%that for iterated Brownian motion $Z_t=B(Y_t)$ the function
%\begin{eqnarray*}
%u(t,x)=E_x[f(Z_t)]:=E[f(Z_t)|Z_0=x]
%\end{eqnarray*}
%solves the initial value problem
%\begin{eqnarray}\label{1.8}
%\begin{split}
%&\frac{\partial}{\partial t}u(t,x)=\Delta^2u(t,x)+\frac{\Delta f(x)}{\sqrt{\pi t}},\quad t>0,\\
%&u(0,x)=f(x),
%\end{split}
%\end{eqnarray}
%for $x\in\mathbb{R}^d$. The non-Markovian property of iterated
%Brownian motion is reflected in the appearance of the time-variable
%initial term in the PDE.
Keyantuo and Lizama \cite{KL} gave the connections between
(\ref{1.3}) with $\alpha = 1/m$ and ordinary non-homogeneous
equations. In Section 4, by analysing the solutions of fractional
Cauchy problems directly we can recover the result in \cite{KL}.
Moreover, we will consider more general fractional Cauchy problem
with the fractional order not necessarily a rational number.

Our work is organized as follows. We provide in Section 2 some
preliminaries of fractional resolvent families and fractional powers
of sectorial operators. And then give positive answers to the
questions $(Q_1)-(Q_5)$ in Section 3 and more results of fractional
generations are obtained as well. Finally, we discuss the relations
of solutions of fractional Cauchy problems and Cauchy problems of
first order in Section 4.

%%%%%%%%%%%%%%%%%%%%%%%%%%%%%%%%%%%%%%%%%%%%%%%%%%%%%%%%%%%%%%%%%%%%%

\section{Preliminaries}

Throughout the paper, $(X,\,\|\cdot\|)$ is a complex Banach space,
and $B(X)$ is the space of all bounded linear operators on $X$. $A$
is a closed linear operator on $X$. We assume throughout this paper
that $A$ is densely defined. By $D(A),R(A),\rho(A),\sigma(A)$ and
$R(\lambda,A)\,(\lambda\in\rho(A))$ we denote the domain, range,
resolvent set, spectrum set and resolvent of the operator $A$,
respectively.

Recall the Caputo fractional derivative of order $\alpha>0$
\begin{eqnarray*}
D_t^\alpha f(t):=J_t^{m-\alpha}\frac{d^m}{dt^m}f(t),
\end{eqnarray*}
where $m$ is the smallest integer greater than or equal to $\alpha$,
and the Riemann-Liouville fractional integral of order $\beta>0$
\begin{eqnarray*}
J_t^\beta f(t)=g_\beta*f(t):=\int_0^tg_\beta(t-s)f(s)ds,
\end{eqnarray*}
where
\newcommand\dif{\mathrm{d}}
            \[
            g_\beta(t):=
            \begin{cases}
            \frac{t^{\beta-1}}{\Gamma(\beta)}, &\quad t>0;\\
            0, &\quad t\leq 0.
            \end{cases}
            \]
Set moreover $g_0(t):=\delta(t)$, the Dirac delta-function. For
details in fractional calculus, we refer the reader to \cite{KST,
Po}.

The Mittag-Leffler function is defined by
\begin{equation}\label{Mittag}
E_{\alpha, \beta}(z) : = \sum_{n=0}^\infty \frac{z^n}{\Gamma(\alpha
n + \beta)}
 = \frac{1}{2\pi i} \int_C \frac{ \mu^{\alpha - \beta} e^\mu}{\mu^\alpha -z} \,d\mu,
 \quad \alpha,\beta>0, z \in \mathbb{C},
 \end{equation}
where the path $C$ is a loop which starts and ends at $-\infty$, and
encircles the disc $|t| \le |z|^{1/\alpha}$ in the positive sense.
$E_{\alpha}(z) := E_{\alpha, 1}(z)$. The Mittag-Leffler function
$E_\alpha(t)$ satisfies the fractional differential equation
$$
D_t^\alpha E_\alpha(\omega t^\alpha)= \omega E_\alpha(\omega
t^\alpha).
$$
The most interesting properties of the Mittag-Leffler functions are
associated with their Laplace integral
\begin{equation}\label{Elaplace}
\int_0^\infty e^{-\lambda t} t^{\beta -1} E_{\alpha, \beta}(\omega
t^\alpha) \,dt = \frac{\lambda^{\alpha - \beta}}{\lambda^\alpha -
\omega}, \quad Re \lambda > \omega^{1/\alpha}, \omega>0
\end{equation}
and with their asymptotic expansion as $z \to \infty$. If $0< \alpha
<2$, $\beta >0$, then
\begin{equation}\label{halfalpha}
E_{\alpha, \beta}(z) = \frac{1}{\alpha}
z^{(1-\beta)/\alpha}\exp(z^{1/\alpha})+ \eps_{\alpha,\beta}(z),
\quad |\arg z| \le \frac{1}{2}\alpha \pi,
\end{equation}
\begin{equation}\label{Eestimate}
 E_{\alpha,\beta}(z) = \eps_{\alpha,\beta}(z), \quad |\arg(-z)| < (1-\frac{1}{2}\alpha)\pi,
\end{equation}
where
$$\eps_{\alpha,\beta}(z)= - \sum_{n=1}^{N-1} \frac{z^{-n}}{\Gamma(\beta - \alpha n)} + O(|z|^{-N})$$
as $z \to \infty$, and the $O$-term is uniform in $\arg z$ if
$|\arg(-z) | \le (1- \alpha/2 - \epsilon)\pi$. It is also of
interest to know the relations between the Mittag-Leffler function
and function of Wright type:
$$
E_\gamma(z) = \int_0^\infty \Psi_\gamma(t)e^{z t} dt, \quad z \in
\Complex, \, 0<\gamma<1,
$$
where
\begin{equation}\label{wright}
\Psi_\gamma(z):= \sum_{n=0}^\infty \frac{(-z)^n}{n!\Gamma(-\gamma n
+1-\gamma)}= \frac{1}{2\pi i} \int_\Gamma \mu^{\gamma
-1}\exp(\mu-z\mu^\gamma)d\mu
\end{equation}
with $\Gamma$ a contour which starts and ends at $-\infty$ and
encircles the origin once counterclockwise. For more properties of
the Mittag-Leffler function and function of Wright type, we refer to
\cite{GM, GLM}.

 We now turn to a short
introduction to fractional powers of sectorial operators. Let $A$ be
a densely defined closed linear operator on Banach space $X$.

\begin{defn}\label{sectorial operator}
The operator $A$ is called sectorial of angle $\omega \in [0, \pi)$
($A\in$ Sect$(\omega)$, in short) if

 1) $\sigma(A)$ is contained in the closure of the sector
$$
\Sigma_\omega: = \{z \in \mathbb{C}: z \not= 0 \mbox{ and }|\arg z|
< \omega\},
$$
for $\omega >0$ or $\Sigma_0: =(0,\infty)$.

 2) For every $\omega' \in (\omega, \pi)$, $\sup\{\|zR(z,A)\|: z \in \mathbb{C} \backslash \overline{\Sigma_{\omega'}}
 \} < \infty$.

 A family of operators $(A_\tau)_{\tau \in \Lambda}$ is called uniformly sectorial of angle $\omega \in [0, \pi)$ if $A_\tau \in$
 Sect$(\omega)$ for each $\tau$, and $\sup\{\|zR(z,A_\tau)\|: \tau \in \Lambda, z \in \mathbb{C} \backslash \overline{\Sigma_{\omega'}}
 \} < \infty$.
\end{defn}

If $0\in\rho(A)$ for a sectorial operator $A$, then we can define
its fractional powers as follows. For $b>0$, define $A^{-b}$ by
\begin{eqnarray}\label{2.2}
A^{-b}:=-\frac{1}{2\pi
i}\int_{\Gamma(\zeta)}\lambda^{-b}R(\lambda,A)d\lambda,
\end{eqnarray}
where the path $\Gamma(\zeta)$ runs in the resolvent set of $A$ from
$\infty e^{-i\zeta}$ to $\infty e^{i\zeta}$, while avoiding the
negative real axis and the origin, and $\lambda^b$ is taken as the
principle branch. Noticing that $A^{-b}\in B(X)$ is injective for
all $b>0$, we can define $A^b:=(A^{-b})^{-1}$ and $A^0:=I$. On the
other hand, for a sectorial operator $A$ without the assumption that
$0\in \rho(A)$, since $A+\epsilon$ is sectorial and
$0\in\rho(A+\epsilon)$, it makes sense to consider the operator
$(A+\epsilon)^b$ and define the fractional powers of $A$ by
\begin{eqnarray*}
A^b:=s-\lim_{\epsilon\rightarrow0^+}(A+\epsilon)^b
\end{eqnarray*}
for $b>0$ and so corresponding results for such fractional powers
can be obtained by similar argument (cf. \cite{Ha, MS}).  We collect
some basic properties of fractional powers in the following lemma.

\begin{lem}\cite{Ha}\label{Sectorial operator lem}
Let $b>0$ and $A^{-b}$ is defined as above. The following assertions
hold.

(a) $A^b$ is closed and $D(A^b)\subset D(A^c)$ for $b>c>0$.

(b) $A^bx=A^{b-n}A^nx$ for all $x\in D(A^n)$ and
$n>b,n\in\mathbb{N}$.

(c) Let $d>b>0$. If $B\subset A^b$ and $D(B)=D(A^d)$, then $B$ is
closable and $\overline{B}=A^b$, where $\overline{B}$ is the closure
of $B$.

(d) $\sigma(A^b)=(\sigma(A))^b$.

(e) If $A \in sect(\omega)$ for some $\omega \in (0, \pi)$, then for
every $\beta \in (0,\pi/\omega)$ the operator $A^\beta$ is sectorial
of angle $\beta \omega$.

(f) If $A \in sect(\omega)$ for some $\omega \in (0, \pi)$, then the
family $(A+\eps)_{\eps \ge 0}$ is uniformly sectorial of angle
$\omega$.
\end{lem}

Finally we recall the notion of $\alpha$-times resolvent families.
Also here we suppose that $A$ is a densely defined closed linear
operator on $X$.

\begin{defn}\label{resolvent family}
Let $\alpha>0$. A family $\{S_{\alpha}(t)\}_{t \geq 0} \subset B(X)$
is called an $\alpha$-times resolvent family generated by $A$ if the
following conditions are satisfied:

(a) $S_\alpha (t)$ is strongly continuous for $t \geq 0$ and
$S_\alpha (0)=I$;

(b) $S_\alpha (t) A \subset A S_\alpha (t)$ for $t \geq 0$;

(c) for $x \in D(A)$, the resolvent equation
\begin{eqnarray}\label{resolvent equation}
S_{\alpha}(t)x = x + \int_0^t g_\alpha (t-s) S_{\alpha}(s)Ax ds
\end{eqnarray}
holds for all $t\geq 0$.
\end{defn}

\begin{rem}\label{changeorder}
Since $A$ is densely defined and closed, it is easy to show that for
all $x \in X$, $\int_0^t g_\alpha (t-s) S_{\alpha}(s)x ds\in D(A)$
and $S_{\alpha}(t)x = x + A(\int_0^t g_\alpha (t-s) S_{\alpha}(s)x
ds)$.
\end{rem}

\begin{defn}
(a) An $\alpha$-times resolvent family $\{S_{\alpha}(t)\}_{t \geq
0}$ is said to be bounded if there exist constants $M\geq1$ such
that $ \|S_{\alpha}(t)\|\leq M$ for all $t\geq0$. If $A$ generates a
bounded $\alpha$-times resolvent family $S_\alpha$, we will write
$(A,S_{\alpha})\in \mathcal{C}_\alpha(0)$ or $A\in
\mathcal{C}_\alpha(0)$ for short.

(b) Let $\theta_0 \in (0,\pi/2]$ and $\omega_0 \geq 0$. An
$\alpha$-times resolvent family $\{S_{\alpha}(t)\}_{t \geq 0}$ is
called analytic of angle $\theta_0$ for some $\theta_0 \in
(0,\pi/2]$ if $S_\alpha(t)$ admits an analytic extension to the
sector $\Sigma_{\theta_0}$. An analytic $\alpha$-times resolvent
family $\{S_{\alpha}(z)\}_{z \in \Sigma_{\theta_0}}$  is said to be
bounded if for each $\theta\in(0,\theta_0)$ there exists a constant
$M_\theta$ such that
\begin{eqnarray*}
\|S_\alpha(z)\|\leq M_\theta,\quad z\in\Sigma_\theta.
\end{eqnarray*}
If $A$ generates a bounded analytic  $\alpha$-times resolvent family
$S_\alpha$ of angle $\theta_0$, we will write $(A,S_{\alpha})\in
\mathcal{A}_\alpha(\theta_0)$ or $A\in \mathcal{A}_\alpha(\theta_0)$
for short.
\end{defn}

%For some convenince, we also write $\mathcal{A}_\alpha(0)=
%\mathcal{C}_\alpha(0)$. Let $A \in \mathcal{C}_\alpha (M,\omega)$
%for some $M \geq 1$ and $\omega \geq 0$. Note that if $\alpha > 2$
%then $A \in B(X)$ (cf. Theorem 2.6, \cite{Baj}), so that we mainly
%discuss $\alpha$-times resolvent families for $\alpha \in (0,2]$.
%The following lemmas collect some basic properties of $\alpha$-times
%resolvent families, for the proofs we refer to \cite{Baj,CL}.

\begin{lem}\cite{Baj}\label{Laplace transform}
Let $0 < \alpha \leq 2$. $A\in \mathcal{C}_\alpha(0)$ if and only if
$\Sigma_{\pi\alpha/2}\subset\rho(A)$ and there exists a strongly
continuous function $S_\alpha:\mathbb{R}_+\rightarrow B(X)$ such
that $\|S_\alpha(t)\|\leq M\mbox{ for all }t\geq0$ and
\begin{equation}\label{Laplace}
\lambda^{\alpha-1}(\lambda^\alpha-A)^{-1}x=\int_0^\infty e^{-\lambda
t}S_\alpha(t)xdt,\quad \lambda\in \Sigma_{\pi/2}
\end{equation}
for all $x\in X$. Furthermore, $\{S_\alpha (t)\}_{t \geq 0}$ is the
$\alpha$-times resolvent family generated by $A$.
\end{lem}

In the sequel we need the following important lemma on analyticity
criteria for $\alpha$-times resolvent families.

\begin{lem}\label{analytic criteria}
Let $\alpha\in(0,2)$ and
$\theta_0\in(0,\min\{\frac{\pi}{2},\frac{\pi}{\alpha}-\frac{\pi}{2}\}]$.
The following assertions are equivalent.

(a) $(A,S_\alpha)\in \mathcal{A}_\alpha(\theta_0)$.

(b) $\Sigma_{\alpha(\frac{\pi}{2}+\theta_0)}\in\rho(A)$, and for
each $\theta\in(0,\theta_0)$, there exists a constant $M_\theta$
such that
$$
\|\lambda (\lambda-A)^{-1}\|\leq M_\theta, \quad
\lambda\in\Sigma_{\alpha(\frac{\pi}{2}+\theta)}.
$$

(c) $-A \in$ {\rm Sect}$(\pi-(\frac{\pi}{2}+\theta_0)\alpha)$.

\end{lem}

The equivalence of $(a)$ with $(b)$ is given in \cite{Baj}. $(b)$ is
equivalent to $(c)$ by the definition of sectorial operators, which
is also mentioned in Remark 3 of \cite{KaL}.

\begin{rem}\label{rem-anacri}
(a) By Lemma \ref{analytic criteria}, $-A$ generates a bounded
analytic $\alpha$-times resolvent family if and only if $A$ is
sectorial of angle $\varphi<\pi-\pi\alpha/2$.

(b) If $-A$ generates a bounded $\alpha$-times resolvent family,
then $A$ is sectorial of angle $\pi-\pi\alpha/2$.
\end{rem}

Recall that if $\{S_\alpha(z)\}_{z \in \Sigma_\theta}$ is a bounded
analytic $\alpha$-times resolvent family with generator $A$, then
for $t >0$,
\begin{equation}\label{3.1}
S_\alpha(t)=\frac{1}{2\pi i}\int_{\Gamma_{\theta_0}}e^{\lambda
t}\lambda^{\alpha-1}(\lambda ^\alpha-A)^{-1}d\lambda,
\end{equation}
where $\Gamma_{\theta_0}$ is any piecewise smooth curve in
$\Sigma_{\pi/2 + \theta}$ going from $\infty e^{-i(\pi/2 +
\theta_0})$ to $\infty e^{i(\pi/2 + \theta_0})$ for some $0<\theta_0
< \theta$  (cf. \cite{Baj, CL}).

 The following subordination principle is important in the theory of
 fractional resolvent families, which will be extended to more general cases in Theorem \ref{main}.

 \begin{lem}\cite{Baj}\label{subordinate}
 Let $0<  \beta<\alpha \le 2$, $\gamma = \beta/\alpha$ and $\omega
 \ge 0$. If $A \in \mathcal{C}_\alpha(0)$ then $A \in
\mathcal{C}_\beta(0)$ and the following representation holds
$$
S_\beta(t) = \int_0^\infty \varphi_\gamma(t,s)S_\alpha(s) ds, \quad
t >0,
$$
where $\varphi_\gamma(t,s)= t^{-\gamma}\Phi_\gamma(st^{-\gamma})$
with $\Phi_\gamma$ defined by (\ref{wright}), in the strong sense.
 \end{lem}

%%%%%%%%%%%%%%%%%%%%%%%%%%%%%%%%%%%%%%%%%%%%%%%%%%%%%%%%%%%%%%%%%%%%%

\section{Fractional powers of generators of fractional resolvent families}\label{3}

In this section we consider the fractional generations for bounded
analytic fractional resolvent families. The following theorem is our
main result, which gives the answer to question $(Q_5)$ in the
Introduction.

\begin{thm}\label{main}
Let $\alpha \in (0,2]$ and $A$ be sectorial of angle $\pi -
\frac{\alpha}{2}\pi$ on a Banach space $X$, and let $0<\gamma<2$.

(a) For each $\beta\in(0,\frac{2\pi-\pi\gamma}{2\pi-\pi\alpha})$,
$-A^\beta\in\mathcal{A}_\gamma(\varphi_0)$ with
$\varphi_0=\min\{\frac{\pi}{2},-\frac{\beta}{\gamma}(\pi-\frac{\pi}{2}\alpha)+\frac
{\pi}{\gamma}-\frac{\pi}{2}\}$.

(b) If $0 \in \rho(A)$, then the $\gamma$-times resolvent family
generated by $-A^\beta$, $S_\gamma^\beta$, can be represented by
\begin{eqnarray}\label{3.2}
\begin{split}
S_\gamma^\beta(t)=\frac{1}{2\pi
i}\int_{\Gamma_\omega}E_\gamma(-\mu^\beta
t^\gamma)(A-\mu)^{-1}d\mu,\quad t
>0,
\end{split}
\end{eqnarray}
where $\Gamma_\omega$ is a smooth path in the resolvent of $A$ from
$\infty e^{-i\omega}$ to $\infty e^{i\omega}$, avoiding the negative
axis and zero, with $\omega \in(\pi- {\alpha \over 2}\pi,
\frac{1}{\beta}(\pi-\frac{\gamma}{2}\pi))$.

(c) If in addition $-A$ generates a bounded $\alpha$-times resolvent
family $S_\alpha$, then the following generalized subordination
principle
\begin{eqnarray}\label{subordination principle}
\begin{split}
S_\gamma^\beta(t)x = \int_0^\infty f_{\alpha,\gamma}^\beta (t,s)
S_\alpha (s)xds, \quad t>0,
\end{split}
\end{eqnarray}
holds for $x \in X$, where
\begin{eqnarray}\label{subordination principle f}
f_{\alpha, \gamma}^\beta(t,s) = \frac{1}{2\pi i} \int_{\partial
\Sigma_{\omega}} E_\gamma(-\mu^\beta t^\gamma)(-\mu)^{1/\alpha -1}
e^{-(-\mu)^{1/\alpha} s} d\mu
\end{eqnarray}
with $\omega$ as in (b), $\partial \Sigma_\omega$ is the two rays
$\{\rho e^{\pm i \omega}: \rho \ge 0\}$ and $(-\rho e^{\pm i
\omega})^{1/\alpha} = \rho^{1/\alpha} e^{\mp i
(\pi-\omega)/\alpha}$.
\end{thm}

\begin{proof}
(a) Since $A$ is sectorial of angle $\pi-\frac{\pi}{2}\alpha$, by
Lemma \ref{Sectorial operator lem} $(e)$, $A^{\beta}$ is sectorial
of angle $\beta(\pi-\frac{\pi}{2}\alpha)$ for $\beta \in (0,
\frac{2\pi}{2\pi-\pi\alpha})$. By Lemma \ref{analytic criteria},
$-A^\beta \in \mathcal{A}_\gamma (\varphi_0)$ if and only if
$A^\beta \in {\rm Sect}(\pi - ({\pi\over 2} + \varphi_0)\gamma)$. To
guarantee that $\varphi_0
>0$, we need $\beta < \frac{2\pi - \pi \gamma}{2\pi - \pi\alpha}$.

(b) Since $A^\beta \in $ Sect$(\beta(\pi-\frac{\pi}{2}\alpha))$,
$\rho(A^\beta) \supset \Complex -
\Sigma_{\beta(\pi-\frac{\pi}{2}\alpha)}$. Thus $(\lambda +
A^\beta)^{-1}$ exists and belongs to $B(X)$ for $\lambda \in
\Sigma_{\pi - \beta(\pi-\frac{\pi}{2}\alpha)}$. Let $ \omega> \pi -
\frac{\pi}{2}\alpha$. Since $0\in \rho(A)$, we can find $d>0$ such
that $\{z \in \Complex: |z|< d\} \subset \rho(A)$ and then choose
$\Gamma_\omega$ as the union of $\Gamma_\omega^1$, $\Gamma_\omega^2$
and $\Gamma_\omega^3$, where
\begin{eqnarray*}
\Gamma_\omega^1&=& \{re^{i\omega}:\, r>d\},\\
\Gamma_\omega^2&=& \{de^{i\theta}:\, -\omega < \theta< \omega\},\\
\Gamma_\omega^3&=& \{re^{-i\omega}:\, r>d\}.\\
\end{eqnarray*}
For $\lambda \in \Sigma_{
\pi-\beta \omega}$, the function $f(\mu) =\frac{1}{\lambda +
\mu^\beta}$ is analytic on $\Sigma_{\omega}$, we can therefore
define a bounded operator $f(A)$ as
$$
f(A) = \frac{1}{2\pi i} \int_{\Gamma_\omega} f(\mu)(A-\mu)^{-1}d\mu.
$$
Since $ \beta(\pi-\omega) < \beta(\pi- \frac{\pi}{2}\alpha)$,
$(\lambda + A^\beta)^{-1} \in B(X)$ for $\lambda \in
\Sigma_{\pi-\beta \omega}$. It is routine to show that for such
$\lambda \in \Sigma_{\pi-\beta \omega}$,
\begin{equation}\label{Abeta}
(\lambda + A^\beta)^{-1} = f(A) = \frac{1}{2\pi i}
\int_{\Gamma_\omega} (\lambda + \mu^\beta)^{-1} (A-\mu)^{-1}d\mu.
\end{equation}
Now take $\beta$ and $\varphi_0$ as in (a). Since $-A^\beta \in
\mathcal{A}_\gamma(\varphi_0)$, for $0< \delta < \varphi_0$, choose
$\omega<  \frac{1}{\beta}[\pi - (\frac{\pi}{2} + \delta)\gamma]$
such that when $\lambda \in \Gamma_{\frac{\pi}{2}+ \delta}$ then
$\lambda^\gamma \in \Sigma_{ \pi-\beta \omega}$. Thus by
(\ref{3.1}), (\ref{Abeta}) and Fubini's theorem we have
\begin{eqnarray*}
S_\gamma^\beta(t) &=& \frac{1}{2\pi i} \int_{\Gamma_{\frac{\pi}{2}+
\delta} }e^{\lambda t}
\lambda^{\gamma -1} (\lambda^\gamma+A^\beta)^{-1} d\lambda \\
&=& \frac{1}{2\pi i} \int_{\Gamma_{\frac{\pi}{2}+ \delta} }
e^{\lambda t }\lambda^{\gamma-1}\Big(\frac{1}{2\pi i}
\int_{\Gamma_\omega} (\lambda^\gamma + \mu^\beta)^{-1} (A-\mu)^{-1}
d\mu \Big)d\lambda
\\
&=&\frac{1}{2\pi i} \int_{\Gamma_\omega} \Big(\frac{1}{2\pi
i}\int_{\Gamma_{\frac{\pi}{2}+ \delta} }e^{\lambda t
}\lambda^{\gamma-1} (\lambda^\gamma + \mu^\beta)^{-1}d\lambda
\Big)(A-\mu)^{-1} d\mu
\\
&=&\frac{1}{2\pi i} \int_{\Gamma_\omega} E_\gamma(-\mu^\beta
t^\gamma)(A-\mu)^{-1} d\mu.
\end{eqnarray*}

(c) We first assume that $0 \in \rho(A)$ and $A \in
\mathcal{C}_\alpha(0)$. Let $\Gamma_\omega$ be as in (b). By (b),
(\ref{Laplace}) and Fubini's theorem, for $x \in X$,
\begin{eqnarray*}
S_\gamma^\beta(t)x &=&\frac{1}{2\pi i} \int_{\Gamma_\omega}
E_\gamma(-\mu^\beta t^\gamma)(A-\mu)^{-1}x d\mu\\
&=& \frac{1}{2\pi i} \int_{\Gamma_\omega^1\cup\Gamma_\omega^3}
E_\gamma(-\mu^\beta t^\gamma)\Big[-(-\mu)^{1/\alpha-1} \int_0^\infty
e^{-(-\mu)^{1/\alpha} s} S_\alpha(s)x ds\Big] d\mu\\
&\mbox{}&\quad +\frac{1}{2\pi i} \int_{\Gamma_\omega^2}
E_\gamma(-\mu^\beta t^\gamma)(A-\mu)^{-1}x d\mu \\
 &=&\int_0^\infty\Big[-\frac{1}{2\pi i}
\int_{\Gamma_{\omega}^1 \cup \Gamma_{\omega}^3} E_\gamma(-\mu^\beta
t^\gamma)(-\mu)^{1/\alpha-1} e^{-(-\mu)^{1/\alpha} s}d\mu
\Big] S_\alpha(s)xds\\
&\mbox{}& \quad +\frac{1}{2\pi i} \int_{\Gamma_\omega^2}
E_\gamma(-\mu^\beta t^\gamma)(A-\mu)^{-1} d\mu
\end{eqnarray*}
The integration on $\Gamma_\omega^2$ converges to 0 if $d \to 0$
since $0 \in \rho(A)$. Moreover, since $|\arg(\mu^\beta t^\gamma)| <
\pi- \frac{\gamma}{2}\pi$, by (\ref{Eestimate}) the integration on
$\Gamma_\omega^1\cup \Gamma_\omega^3$ is absolutely convergent if
$d\to 0$. So letting $d \to 0$, we get $S_\gamma^\beta(t)x =
\int_0^\infty f_{\alpha,\gamma}^\beta(t,s)S_\alpha(s)x ds$ with
\begin{eqnarray*}
f_{\alpha, \gamma}^\beta(t,s) =-\frac{1}{2\pi i} \int_{\partial
\Sigma_\omega} E_\gamma(-\mu^\beta t^\gamma)(-\mu)^{1-\alpha\over
\alpha} e^{-(-\mu)^{1\over\alpha} s}   d\mu.
%&=&\frac{1}{2\pi i} \int_{\partial \Sigma_{\xi}}
%E_\gamma(-(-\mu^\alpha)^\beta t^\gamma) e^{-\mu s} d\mu
\end{eqnarray*}
%with $\xi \in ((\frac{1}{\alpha}-\frac{1}{\alpha \beta} +
%\frac{\gamma}{2\alpha \beta})\pi,\pi/2)$.
Thus (\ref{subordination
principle}) holds for $-A\in \mathcal{C}_\alpha(0)$ with $0 \in
\rho(A)$.

Next we show that (\ref{subordination principle}) holds when $0
\not\in \rho(A)$ and $-A$ generates a bounded analytic
$\alpha$-times resolvent family. For $\eps >0$, $0 \in \rho(A+\eps)$
and $(A+\eps)_{\eps \ge 0}$ is uniformly sectorial of angle $\pi -
\frac{\alpha}{2}\pi$ by Lemma \ref{Sectorial operator lem}(f). By
(b), the $\gamma$-times resolvent family, $\mbox{}_\eps
S^\beta_\gamma$, generated by $-(A+\eps)^\beta$ is given by
\begin{equation}\label{Seps}
\mbox{}_\eps S_\gamma^\beta(t)=\frac{1}{2\pi
i}\int_{\Gamma_\omega}E_\gamma(-\mu^\beta
t^\gamma)(A+\eps-\mu)^{-1}d\mu,\quad t
>0,
\end{equation}
since $(A+\eps -\mu)^{-1} \to (A-\mu)^{-1}$ as $\eps \to 0$, by
(\ref{Eestimate}) and Lebesgue's dominated convergence theorem
$\mbox{}_\eps S_\gamma^\beta(t)\to S_\gamma^\beta(t)$ as $\eps \to
0$ for every $t \ge 0$. On the other hand, by the first step we can
represent $\mbox{}_\eps S_\gamma^\beta(t)$ by
\begin{equation}\label{Seps2}
\mbox{}_\eps S_\gamma^\beta(t)x = \int_0^\infty
f_{\alpha,\gamma}^\beta (t,s) \mbox{}_\eps S_\alpha (s)xds, \quad
t>0,
\end{equation}
where $\mbox{}_\eps S_\alpha (s)$ is the $\alpha$-times resolvent
family generated by $-(A+\eps)$. Since $\mbox{}_\eps S_\alpha (s)$
is uniformly bounded and $(A+\eps -\mu)^{-1} \to (A-\mu)^{-1}$ as
$\eps \to 0$, by the approximation theorem for $\alpha$-times
resolvent family (Theorem 4.2 in \cite{LZ}) one has $\mbox{}_\eps
S_\alpha (s) \to S_\alpha (s)$ in strong sense for every $s \ge 0$.
Note that $f_{\alpha,\gamma}^\beta(t,\cdot)$ is absolutely
integrable by (\ref{Eestimate}), by letting $\eps$ to 0 in
(\ref{Seps2}) we obtain (\ref{subordination principle}).

Finally we show that (\ref{subordination principle}) holds when $0
\not\in \rho(A)$ and $-A \in \mathcal{C}_\alpha(0)$. For every
$\alpha'< \alpha$, $-A$ generates a bounded analytic $\alpha'$-times
resolvent family by (a) or Lemma \ref{subordinate}, so by our second
step we have for every $x \in X$,
$$
S_\gamma^\beta(t)x = \int_0^\infty f_{\alpha',\gamma}^\beta (t,s)
S_{\alpha'} (s)xds, \quad t>0,
$$
where $S_{\alpha'}$ is the $\alpha'$-times resolvent family
generated by $-A$. Since $S_{\alpha'}(t) \to S_\alpha(t)$ strongly
by Theorem 4.5 in \cite{LZ} and $f_{\alpha',\gamma}^\beta(t,s) \to
f_{\alpha,\gamma}^\beta(t,s)$, (\ref{subordination principle}) is
obtained by letting $\alpha'$ to $\alpha$.
\end{proof}

\begin{rem}\label{main-rem}
(a) Note that by Remark \ref{rem-anacri} (b), if $A \in
\mathcal{C}_\alpha(0)$, then $A$ is sectorial of angle $\pi - \alpha
\pi/2$.

(b) If $\alpha=1$, we can shift the contour in (\ref{subordination
principle f}) to $\Gamma_{\omega}$, that is,
$$
f_{1,\gamma}^\beta(t,s) =\frac{1}{2\pi i} \int_{\Gamma_\omega}
E_\gamma(-\mu^\beta t^\gamma) e^{\mu s} d\mu.
$$
%And moreover, if $\gamma \le 1$, by (\ref{Eestimate}) one can deform
%the path $\Gamma_{\omega}$ to a vertical line, thus
%$\widehat{f_{1,\gamma}^\beta (t,\cdot)}(\lambda)=E_\gamma
%(-\lambda^\beta t^\gamma)$ for $\gamma \le 1$.

(c) If $\beta=1$, then in the proof of (b) we do not need the
assumption that $0 \in \rho(A)$. Indeed, in this case we can replace
the contour $\Gamma_\omega$ by
$\tilde{\Gamma}_\omega:=\Gamma_\omega^1 \cup \Gamma_\omega^3 \cup
\tilde{\Gamma}_\omega^2$, where $\tilde{\Gamma}_\omega^2 =
\{de^{i\theta}:\, \omega < \theta< 2\pi-\omega\}$, and then
$$
S_\gamma^1(t)=\frac{1}{2\pi
i}\int_{\tilde{\Gamma}_\omega}E_\gamma(-\mu
t^\gamma)(A-\mu)^{-1}d\mu,\quad t
>0,
$$
and
$$
f_{\alpha, \gamma}^1(t,s) =-\frac{1}{2\pi i}
\int_{\tilde{\Gamma}_\omega} E_\gamma(-\mu
t^\gamma)(-\mu)^{1-\alpha\over \alpha} e^{-(-\mu)^{1\over\alpha} s}
d\mu.
$$
In particular, if $(A,S_\alpha)\in\mathcal{A}_\alpha(\theta_0)$ with
$\theta_0
>0$, then for each $\theta\in(0,\theta_0)$ and $z\in\Sigma_\theta$,
$S_\alpha(z)$ has the following integrated representation:
\begin{eqnarray}\label{3.15}
\begin{split}
S_\alpha(z)=\frac{1}{2\pi i}\int_{\tilde{\Gamma}_\omega}E_\alpha(\mu
z^\alpha)(\mu-A)^{-1}d\mu,
\end{split}
\end{eqnarray}
where $\theta\in(\pi\alpha/2, (\pi/2+\theta_0)\alpha)$. Note that
(\ref{3.15}) is a Dunford integral, sometimes it will be more
convenient than the identity (\ref{3.1}).

(d) If $\gamma=1$, by changing the variable $\mu$ in
(\ref{subordination principle f}) to $\rho e^{i \omega}$ and $\rho
e^{-i \omega}$, $0< \rho< \infty$, one gets
%\begin{eqnarray*}
%f_{\alpha,\gamma}^\beta(t,s) &=& \frac{1}{2\pi i}\int_0^\infty
%\rho^{({1-\alpha})/{\alpha}}\exp\Big(-s\rho^{1/\alpha}\cos(\pi-\theta)/{\alpha}\Big)\\
%&\mbox{}& \quad \cdot \Big[E_\gamma(-\rho^\beta e^{-i
%\beta\theta}t^\gamma)\exp\Big(i(-s\rho^{1/\alpha}\sin({\pi-\theta})/{\alpha}
%+ ({\pi-\theta})/{\alpha})\Big) \\
%&\mbox{}& \quad\quad -E_\gamma(-\rho^\beta e^{i
%\beta\theta}t^\gamma)\exp\Big(i(s\rho^{1/\alpha}\sin({\pi-\theta})/{\alpha}
%- ({\pi-\theta})/{\alpha})\Big)\Big] d\rho.
%\end{eqnarray*}
\begin{eqnarray*}
f_{\alpha,1}^\beta(t,s) &=& \frac{1}{\pi }\int_0^\infty
\rho^{({1-\alpha})/{\alpha}}\exp\Big(-s\rho^{1/\alpha}\cos(\pi-\omega)/{\alpha}-t\rho^\beta\sin\beta \omega \Big)\\
&\mbox{}& \quad \cdot \sin\Big(t\rho^\beta \sin \beta \omega
-s\rho^{1/\alpha}\sin({\pi-\omega})/{\alpha} +
({\pi-\omega})/{\alpha})\Big) d\rho.
\end{eqnarray*}

\end{rem}

 As consequences of Theorem
\ref{main} and Remark \ref{main-rem} we have the following results,
which give positive answers to questions $(Q_1)-(Q_4)$.

\begin{cor}\label{bounded-cor}
The following assertions hold.

(a) If $(-A,S_1) \in \mathcal{C}_1(0)$ then
$-A^\alpha\in\mathcal{A}_1(\frac{\pi}{2}(1-\alpha))$ for each
$\alpha\in(0,1)$. Moreover, the $C_0$-semigroup generated by
$-A^{\alpha}$ is given by
$$
 \int_0^\infty p_\alpha (t,s) S_1 (s)ds,\quad t >0
$$
where $\widehat{p_\alpha (t,\cdot)}(\lambda):= \int_0^\infty
e^{-\lambda s} p_\alpha(t, s)ds= e^{-\lambda^\alpha t}$ and
\begin{equation}\label{Yosida}
\begin{split}
p_\alpha(t,s) &= \frac{1}{\pi}  \int_0^\infty \exp(s \rho
\cos\theta-t\rho^\alpha \cos\alpha\theta) \cdot \sin(s \rho
\sin\theta\\
&\mbox{}\quad - t\rho \sin\alpha\theta + \theta)d\rho,
\end{split}
\end{equation}
for $\pi/2 < \theta <\pi$.

 (b) If $(-A,S_\alpha) \in \mathcal{C}_\alpha(0)$ then
$-A\in\mathcal{A}_\beta(\min\{\pi,(\alpha/\beta-1)\pi/2\})$ for each
$\beta\in(0,\alpha)$. Moreover, the $\beta$-times resolvent family
generated by $-A$ is given by
$$
 \int_0^\infty \varphi_{\beta/\alpha} (t,s) S_\alpha(s)ds, \quad t
 >0
$$
where $\widehat{\varphi_\gamma (\cdot,s)}(\lambda)=
\lambda^{\gamma-1} e^{-\lambda^\gamma s}$,
$\widehat{\varphi_\gamma(t, \cdot)} (\lambda) = E_\gamma (-\lambda
t^\gamma)$ for $0<\gamma <1$ and
\begin{eqnarray*}
\varphi_{\gamma}(t,s)&=&
 \frac{1}{\pi}\int_0^\infty \rho^{\gamma -1} \exp\Big(-s\rho^\gamma \cos\gamma({\pi-\theta})-t\rho
 \cos{\theta}\Big)\\
 &\mbox{}&\quad \cdot \sin\Big(t\rho \sin{\theta}-s\rho^\gamma
 \sin\gamma({\pi-\theta}) +\gamma({\pi-\theta})\Big)d\rho
\end{eqnarray*}
for $\theta \in (\pi- \frac{\pi}{2\gamma}, \pi/2)$.

(c) If $(-A,S_1) \in \mathcal{C}_1(0)$ then
$-A^\alpha\in\mathcal{A}_\alpha(\min\{\frac{\pi}{\alpha}-\pi,\pi/2\})$
for each $\alpha\in(0,1)$. Moreover, the $\alpha$-times resolvent
family generated by $-A^{\alpha}$ is given by
\begin{eqnarray*}
\begin{split}
\int_0^\infty f_{1,\alpha}^\alpha (t,s) S_1 (s)ds, \quad t>0
\end{split}
\end{eqnarray*}
where $\widehat{f_{1,\alpha}^\alpha(t,\cdot)}(\lambda)=
E_\alpha(-\lambda^\alpha t^\alpha)$ and $f_{1,\alpha}^\alpha(t,s) =
\int_0^\infty \varphi_\alpha(t,\tau)p_\alpha(\tau,s)d\tau$.

(d) If $(-A,S_\alpha) \in \mathcal{C}_\alpha(0)$ for some $\alpha
\in (1,2]$ then $-A^{1/\alpha}\in\mathcal{A}_1(\pi - \pi/\alpha)$.
Moreover, the $C_0$-semigroup generated by $-A^{1/\alpha}$ is given
by
\begin{eqnarray*}
\begin{split}
\int_0^\infty f_{\alpha,1}^{1/\alpha} (t,s) S_\alpha (s)ds, \quad
t>0
\end{split}
\end{eqnarray*}
where
\begin{eqnarray*}
f_{\alpha,1}^{1/\alpha}(t,s)&=&
 \frac{\alpha}{\pi}\int_0^\infty  \exp\Big(-s\rho \cos({\pi-\theta})/{\alpha}-t\rho
 \cos{\theta}/{\alpha}\Big)\\
 &\mbox{}&\quad \cdot \sin\Big(t\rho \sin{\theta}/{\alpha}-t\rho
 \sin({\pi-\theta})/{\alpha} +({\pi-\theta})/{\alpha}\Big)d\rho
\end{eqnarray*}
for $\theta \in (\pi- \frac{\alpha \pi}{2}, \alpha \pi/2)$ and
%$\widehat{f_{\alpha,1}^{1/\alpha}(t,\cdot)}(\lambda)=
%\lambda^{\frac{1}{\alpha}-1}e^{-\lambda t}$
$f_{\alpha,1}^{1/\alpha}(t,s)= \int_0^\infty
 p_{1/\alpha}(t,\tau)\varphi_{1/\alpha} (\tau,s)d\tau$.

 (e) If $(-A,S_\alpha) \in \mathcal{C}_\alpha(0)$ for some
$\alpha \in (0,2]$ then $-A^{1/2} \in
\mathcal{A}_{\alpha/2}(\pi/2)$. Moreover, the $\alpha/2$-times
resolvent family generated by $-A^{1/2}$ is given by
\begin{eqnarray}\label{subordination, 1/2}
\begin{split}
\frac{\alpha}{\pi} t^{ \frac{\alpha}{2} } \int_0^\infty
\frac{s^{\frac{\alpha}{2} - 1}}{s^\alpha + t^\alpha} S_\alpha (s)ds,
\quad t>0.
\end{split}
\end{eqnarray}

(f) If $(-A,S_\alpha) \in \mathcal{C}_\alpha(0)$ for some $\alpha
\in (0,2)$ then $-A^{\beta} \in
\mathcal{A}_{\alpha}(\min\{(\frac{\pi}{\alpha}-\frac{\pi}{2})(1-\beta),
\pi/2\})$ for $\beta \in (0,1)$.
\end{cor}

\begin{proof}
(a) follows from Remark \ref{main-rem} (a), (b) and (d).

 (b) By Remark
\ref{main-rem} (c), (\ref{Mittag}), Fubini's theorem and Cauchy's
integral formula,
\begin{eqnarray*}
f_{\alpha, \beta}^1(t,s) &=&-\frac{1}{2\pi i}
\int_{\tilde{\Gamma}_\omega} E_\beta(-\mu
t^\beta)(-\mu)^{1-\alpha\over \alpha} e^{-(-\mu)^{1\over\alpha} s}
d\mu\\
&= &\int_{\tilde{\Gamma}_\omega}\Big( \frac{1}{2\pi i}
\int_{\Gamma_{{\pi\over 2}+\delta}} e^{\lambda t}\lambda^{\beta -1}
(\lambda ^\beta +\mu)^{-1} d\lambda \Big)(-\mu)^{1-\alpha\over
\alpha} e^{-(-\mu)^{1\over\alpha} s}
d\mu\\
&=&\frac{1}{2\pi i} \int_{\Gamma_{{\pi\over 2}+\delta}}e^{\lambda
t}\lambda^{\beta -1} \Big(\int_{\tilde{\Gamma}_\omega} (\lambda
^\beta +\mu)^{-1} (-\mu)^{1-\alpha\over \alpha}
e^{-(-\mu)^{1\over\alpha} s}
d\mu\Big)d\lambda \\
&=& \frac{1}{2\pi i} \int_{\Gamma_{{\pi\over 2}+\delta}} e^{\lambda
t}\lambda^{\beta -1}
 \lambda^{\beta(\frac{1-\alpha}{\alpha})} e^{-\lambda^{\beta/\alpha}
 s}d\lambda\\
&=&\varphi_{\beta/\alpha}(t,s).
\end{eqnarray*}
And the last identity follows from Remark \ref{main-rem} (d) and by
noting that $\varphi_\gamma(t,s) = f_{1/\gamma,1}^1(t,s)$.

 (c) By Remark \ref{main-rem} (b), for $\lambda >0$,
\begin{eqnarray*}
\int_0^\infty e^{-\lambda
s}f_{1,\alpha}^\alpha(t,s)ds&=&\int_0^\infty e^{-\lambda
s}\Big(\frac{1}{2\pi i}\int_{\Gamma_\omega} E_\alpha(-\mu^\alpha
t^\alpha) e^{\mu s}
d\mu\Big)ds\\
&=&\frac{1}{2\pi i} \int_{\Gamma_\omega} E_\alpha(-\mu^\alpha
t^\alpha) \Big(\int_0^\infty e^{-\lambda s}e^{\mu s}ds
\Big)d\mu\\
&=&\frac{1}{2\pi i} \int_{\Gamma_\omega} \frac{E_\alpha(-\mu^\alpha
t^\alpha)}{\lambda -\mu}d\mu = E_\alpha(-\lambda^\alpha t^\alpha)
\end{eqnarray*}
holds by Cauchy's integral formula and (\ref{Eestimate}). The last
statement follows from the calculation of the Laplace transform of
the function $\int_0^\infty
\varphi_\alpha(t,\tau)p_\alpha(\tau,\cdot)d\tau$.

(d) The representation of $f_{\alpha,1}^{1/\alpha}$ follows from
Remark \ref{main-rem} (d). By (b), the $C_0$-semigroup generated by
$-A$ is given by $T(t)= \int_0^\infty \varphi_{1/\alpha}(t,s)
S_\alpha(s) ds$; and then by (a), the $(1/\alpha)$-times resolvent
family generated by $-A^{1/\alpha}$ is given by $\int_0^\infty
p_{1/\alpha}(t,s)$ $T(s)ds$.

 (f) and the first part of (e) are immediate consequences of Theorem
\ref{main}. It remains to prove the subordination formulas
(\ref{subordination, 1/2}). Indeed, let $S_{\alpha/2}$ be the
$\alpha/2$-times resolvent family generated by $-A^{1/2}$.
%Note that
%for every $c > 0$, $E_\alpha(-ct^\alpha)$ is bounded
%(scalar valued) $\alpha$-times resolvent family generated by $-c$
%For $\omega > 0$ and $c>0$, by using (\ref{Mittag}) and the complex
%inversion formula one obtain
%\cite{Haase} Proposition 2.1 implies that
%\begin{equation}\label{inverse}
%t^\beta E_{\alpha, \beta+1} (-ct^\alpha) =  \frac{1}{2\pi i}
%\int_{\omega-i\infty}^{\omega+i\infty} e^{\lambda t}
%\frac{\lambda^{\alpha-\beta-1}}{\lambda^\alpha + c} d\lambda,
%\end{equation}
%note that we used the identity $\int_0^t g_\beta (t-s)
%E_\alpha(-cs^\alpha) ds = t^\beta E_{\alpha, \beta+1} (-ct^\alpha)$
%in the above.
Since $\alpha /2 <1$, by  (5.24) in \cite{MS},
\begin{equation}\label{Bafdd}
\bigg( \lambda^{\alpha/2} + A^{1/2} \bigg)^{-1} = \frac{1}{\pi}
\int_0^\infty \frac{\mu^{1/2}}{\mu + \lambda^\alpha}(\mu + A)^{-1}
d\mu.
\end{equation}
Therefore, it follows from (\ref{Laplace}), (\ref{3.1})
%(\ref{inverse})
(\ref{Bafdd}), and Fubini's theorem that
\begin{eqnarray*}
S_{\alpha/2} (t) &=& \frac{1}{2\pi i} \int_{\Gamma_{\theta_0}}
e^{\lambda t} \lambda^{\alpha/2 - 1} \bigg( \lambda^{\alpha/2} +
A^{1/2} \bigg)^{-1} d\lambda\\
&=& \frac{1}{2\pi i} \int_{\Gamma_{\theta_0}} e^{\lambda t}
\lambda^{\alpha/2 - 1} \bigg( \frac{1}{\pi} \int_0^\infty
\frac{\mu^{1/2}}{\mu + \lambda^\alpha}
((\mu^{1/\alpha})^\alpha + A)^{-1} d\mu \bigg) d\lambda\\
&=& \frac{1}{2\pi i} \int_{\Gamma_{\theta_0}} e^{\lambda t}
\lambda^{\alpha/2 - 1} \bigg( \frac{1}{\pi} \int_0^\infty
\frac{\mu^{1/\alpha - 1/2}}{\mu + \lambda^\alpha}
\int_0^\infty e^{-s\mu^{1/\alpha}} S_\alpha(s) ds d\mu \bigg) d\lambda\\
&=& \frac{1}{2\pi i} \int_{\Gamma_{\theta_0}} e^{\lambda t} \bigg(
\frac{\alpha}{\pi} \int_0^\infty S_\alpha(s) ds \int_0^\infty
e^{-s\nu} \frac{\lambda^{\alpha/2 - 1} \nu^{\alpha/2}}{\nu^\alpha +
\lambda^\alpha} d\nu \bigg) d\lambda\\
&=& \frac{\alpha}{\pi} \int_0^\infty S_\alpha (s) ds \int_0^\infty
e^{- s \nu}\Big( \frac{1}{2 \pi i} \int_{\Gamma_{\theta_0}}
e^{\lambda t} \frac{ \lambda^{ \alpha/2 - 1}
\nu^{\alpha/2}}{\nu^\alpha + \lambda^\alpha} d\lambda\Big) d\nu \\
%&=& \frac{\alpha}{\pi} t^{\alpha/2} \int_0^\infty S_\alpha (s) ds
%\int_0^\infty e^{- s \nu} \nu^{\alpha/2} E_{\alpha, \alpha/2 + 1}
%(-\nu^\alpha t^\alpha) d\nu\\
&=& \frac{\alpha}{\pi} t^{\alpha/2} \int_0^\infty
\frac{s^{\frac{\alpha}{2} - 1}}{s^\alpha + t^\alpha} S_\alpha (s)
ds, \quad t > 0.
\end{eqnarray*}

\end{proof}

\begin{rem}\label{integral representation2 cor rem}
(a) (\ref{Yosida}) is the formula (11) in \cite{Yo}. Note that
$p_{1/2}(t,s)= \frac{te^{-t^2/4s}}{2\sqrt{\pi}s^{3/2}}$ (see Lemma
1.6.7 in \cite{ABHN}).

(b) By Corollary \ref{bounded-cor} (b), we obtain the subordinate
principle (Theorem 3.1 in \cite{Baj}) for bounded fractional
resolvent families. The formula is also applied to exponentially
bounded fractional resolvent families by small modification, since
we do not need the fractional power here. By Lemma 1.6.7 in
\cite{ABHN} $\varphi_{1/2}(t,s) = \frac{e^{-s^2/4t}}{\sqrt{\pi t}}$.

(c) By Corollary \ref{bounded-cor} (e), if $A$ generates a bounded
$C_0$-semigroup $T(t)$, then the $1/2$-times resolvent family
generated by $-(-A)^{1/2}$ is given by
$\frac{\sqrt{t}}{\pi}\int_0^\infty \frac{T(s)}{(t+s)s^{1/2}}ds$.

(d) By Corollary \ref{bounded-cor} (e), if $A$ generates a bounded
cosine function $C(t)$, then the $C_0$-semigroup generated by
$-(-A)^{1/2}$ is given by $\frac{2t}{\pi}\int_0^\infty
\frac{C(s)}{t^2+s^2} ds$. See also Lemma 2.1 in \cite{Cio and Liz}.
\end{rem}

The following results for generators of analytic fractional
resolvent families can be proved similarly as the proof of Theorem
\ref{main} (a) by using Lemma \ref{analytic criteria}.

\begin{prop}\label{analytic-cor}
The following assertions hold:

(a) If $-A\in\mathcal{A}_1(\theta_0)$ for some
$\theta_0\in(0,\frac{\pi}{2}]$, then
$-A^\alpha\in\mathcal{A}_1(\frac{\pi}{2}-(\frac{\pi}{2}-\theta_0)\alpha)$
for each $\alpha\in(0,\frac{\pi}{\pi-2\theta_0})$.

(b) If $-A\in\mathcal{A}_1(\theta_0)$ for some
$\theta_0\in(0,\frac{\pi}{2}]$, then
$-A^\alpha\in\mathcal{A}_\alpha(\min\{\frac{\pi}{\alpha}+\theta_0-\pi,\pi/2\})$
for each $\alpha\in(0,\frac{\pi}{\pi-\theta_0})$.

(c) If $-A\in\mathcal{A}_\alpha(\theta_0)$ for some $\alpha \in
(0,2)$ and $\theta_0 \in (0, \min\{\frac{\pi}{2}, \frac{\pi}{\alpha}
- \frac{\pi}{2}\}]$, then
$-A\in\mathcal{A}_\gamma(\min\{\frac{\alpha}{\gamma}(\frac{\pi}{2}+\theta_0)-\frac{\pi}{2},\frac{\pi}{2}\})$
for each $\gamma\in(0, \frac{(\pi+ 2\theta_0)\alpha}{\pi})$.

(d) If $-A\in\mathcal{A}_\alpha(\theta_0)$  for some $\alpha \in
(0,2)$ and $\theta_0 \in (0, \min\{\frac{\pi}{2}, \frac{\pi}{\alpha}
- \frac{\pi}{2}\}]$, then
$-A^{1/\alpha}\in\mathcal{A}_1(-\frac{\pi}{\alpha}+\theta_0+\pi)$ if
$\alpha\in(\frac{\pi}{\pi+\theta_0},2)$.

(e) If $-A\in\mathcal{A}_\alpha(\theta_0)$  for some $\alpha \in
(0,2)$ and $\theta_0 \in (0, \min\{\frac{\pi}{2}, \frac{\pi}{\alpha}
- \frac{\pi}{2}\}]$, then
$-A^{\beta}\in\mathcal{A}_\alpha(\min\{-\frac{\beta}{\alpha}\pi +
\frac{\beta}{2\alpha}\pi + \beta \theta_0 + \frac{\pi}{\alpha} -
\frac{\pi}{2}, \pi/2\})$ if $\beta \in(0,
\frac{(2-\alpha)\pi}{2\pi-(\pi+\theta_0)\alpha})$.
\end{prop}

\begin{rem}\label{integral representation1 cor rem}
Proposition \ref{analytic-cor} (a) improves Theorem 3.1 in \cite{HZ}
in that we do not need $0\in \rho(A)$.
\end{rem}

\begin{example}\label{example1}
Let $\alpha\in(0,2)$ and $k>0$.
% Consider the fractional
%diffusion-wave equation
%\begin{eqnarray*}
%\begin{split}
%D_t^\alpha u=k^2u_{xx},\quad -\infty<x<\infty,\,t>0,
%\end{split}
%\end{eqnarray*}
%with condition $u(\mp\infty,t)=0,u(x,0)=f(x)$ and $u_t(x,0)=0$ (the
%last one only for $\alpha\in(1,2)$).
Set $X:=L^p(\mathbb{R}),\,A:=-kD_x^2$ with $D(A)=W^{2,p}(\Real)$. It
is well known that $-A$ generates a bounded analytic semigroup of
angle $\frac{\pi}{2}$. Thus, by Proposition \ref{analytic-cor} one
has

(a)
$-A\in\mathcal{A}_\alpha(\min\{\frac{\pi}{\alpha}-\frac{\pi}{2},\frac{\pi}{2}\})$
for all $\alpha\in(0,2)$;

(b)
$-A^\alpha\in\mathcal{A}_\alpha(\min\{\frac{\pi}{\alpha}-\frac{\pi}{2},\frac{\pi}{2}\})$
for all $\alpha\in(0,2)$;

(c) $-A^\alpha \in \mathcal{A}_1(\pi/2)$ for all $\alpha \in
(0,\infty)$.

%On the other hand, by Lemma \ref{analytic criteria}, it follows that
%$-A\in\mathcal{A}_\alpha(\min\{\frac{\pi}{\alpha}-\frac{\pi}{2},\frac{\pi}{2}\})$
%for each $\alpha\in(0,2)$. Consequently, by  Corollary
%\ref{analytic-cor} (d),
%$-A^{\frac{1}{\alpha}}\in\mathcal{A}_1((\frac{3}{2}-\frac{1}{\alpha})\pi)$
%for all $\alpha\in(\frac{2}{3},2)$.
\end{example}

\begin{example}\label{example2}
Let $\alpha\in(0,2)$ and $\theta\in[0,\pi)$.
%Consider the problem
%\begin{eqnarray*}
%\begin{split}
%D_t^\alpha u=e^{i\theta}u_{xx},\quad 0<x<1,t>0,
%\end{split}
%\end{eqnarray*}
%with initial and boundary conditions $u(0,t)=u(1,t),u(x,0)=f(x)$ and
%$u_t(x,0)=0$ (the second only for $\alpha\in(1,2)$).
Set $X:=L^2(0,1),\,B_\theta:=-e^{i\theta}D_x^2$ with
$D(B_\theta)=\{f\in W^{2,2}(0,1):f(0)=f(1)=0\}$. It is proved that
for $\frac{\pi}{2}<\theta\leq (1-\frac{\alpha}{2})\pi$, $-B_\theta
\in \mathcal{A}_\alpha(\theta_0)$ with
$\theta_0=\min\{\frac{\pi}{\alpha}-\frac{\theta}{\alpha}-\frac{\pi}{2},\frac{\pi}{2}\}$,
but does not generate any $C_0$-semigroup (see Example 2.20 in
\cite{Baj}). However, by Corollary \ref{bounded-cor} (d),
$-B_\theta^{1/\alpha} \in \mathcal{A}_1({\pi\over 2} - {\theta\over
\alpha} )$ for $\alpha \in (1,2)$.
\end{example}

%%%%%%%%%%%%%%%%%%%%%%%%%%%%%%%%%%%%%%%%%%%%%%%%%%%%%%%%%%%%%%%%%%%%%

\section{Solutions to  fractional Cauchy problems}\label{4}
In this section we will consider the solutions of fractional Cauchy
problems. First we give the definitions of solutions to the
inhomogeneous initial value problem
\begin{equation}\label{ACPf}
\begin{split}
&D_t^\alpha u(t) =Au(t) + f(t),\,\, t \in (0,\tau) \\
&u^{(k)}(0)=x_k,\,\, k=0,1,\cdots,m-1,
\end{split}
\end{equation}
where $\tau\in (0,+\infty]$, $f \in L^1_{loc}([0,\tau);X)$ and $A$
is a closed densely defined operator on Banach space $X$.
\begin{defn}
A function $u(t)\in C([0,\tau);X)$ is called a strong solution (or
simply solution) of (\ref{ACPf}) if $u(t)$ satisfies:

 (a) $u(t) \in C([0,\tau); D(A))
\cap C^{m-1}([0,\tau); X)$.

 (b) $g_{m-\alpha}*(u-
\sum\limits_{k=0}^{m-1}g_{k+1}x_k) \in C^{m}([0,\tau);X)$.

 (c) $u(t)$ satisfies Eq. (\ref{ACPf}).

 $u(t) \in C([0,\tau);X)$ is called a mild solution
of (\ref{ACPf}) if $g_\alpha *u \in D(A)$ and
$$u(t)=
\sum_{k=0}^{m-1} g_{k+1}(t)x_k +A(g_{\alpha}* u)(t) +
(g_{\alpha}*f)(t), \quad t \ge 0.
$$
\end{defn}

Suppose that $A$ generates an $\alpha$-times resolvent family
$S_\alpha(t)$, then the strong solution of (\ref{ACPf}) with $f=0$
and $x_k \in D(A)$ is given by
$$
u(t)= \sum_{k=0}^{m-1} (g_k*S_\alpha)(t)x_k,
$$
see \cite{Baj} for more details. So we now turn to the following
problem
\begin{equation}\label{ACPf0}
\begin{split}
&D_t^\alpha u(t) = Au(t) + f(t),\,\, t \in (0,\tau) \\
&u^{(k)}(0)=0,\,\, k=0,1,\cdots,m-1.
\end{split}
\end{equation}
If $u$ is a mild solution of (\ref{ACPf0}), then $g_\alpha
*u \in D(A)$ and $u=A(g_\alpha *u) + g_\alpha *f$. By Remark
\ref{changeorder},
\begin{eqnarray*}
1*u &=& (S_\alpha -A( g_\alpha * S_\alpha))*u= S_\alpha *u - S_\alpha * A(g_\alpha *u)\\
&=& S_\alpha *u - S_\alpha *u + S_\alpha * g_\alpha *f = g_\alpha
*S_\alpha * f,
\end{eqnarray*}
which means that $g_\alpha *S_\alpha *f$ is differentiable and the
mild solution is given by
\begin{equation}\label{mild}
u(t) = \frac{d}{dt}(g_{\alpha }*S_\alpha*f)(t), \quad t \ge 0.
\end{equation}
Consequently we have

\begin{prop}\label{strong-solution}
Let $A$ be the generator of an $\alpha$-times resolvent family
$S_\alpha$ and let $f \in L^1_{loc}([0,\tau);X)$. If (\ref{ACPf0})
has a mild solution, then it is given by (\ref{mild}). And the mild
solution of (\ref{ACPf}) is given by
$$
u(t) =\sum_{k=0}^{m-1}
(g_{k}*S_\alpha)(t)x_k+\frac{d}{dt}(g_{\alpha}*S_\alpha*f)(t), \quad
t \ge 0.
$$
\end{prop}

For the strong solutions of (\ref{ACPf0}), we have
\begin{prop}
Let $\alpha \in (0,2]$. Suppose that $A$ is the generator of an
$\alpha$-times resolvent family $S_\alpha$ and let $f \in
C([0,\tau);X)$. Then the following statements are equivalent:

$(a)$ (\ref{ACPf0}) has a strong solution on $[0,\tau)$.

$(b)$ $S_\alpha *f $ is differentiable on $[0,\tau)$.

$(c)$ $\frac{d}{dt}(g_{\alpha} * S_\alpha *f)(t) \in D(A)$ for $t
\in [0,\tau)$ and $A(\frac{d}{dt}(g_{\alpha} * S_\alpha *f)(t))$ is
continuous on $[0,\tau)$.

In the case $\alpha \in [1,2]$, the condition $(c)$ can be replaced
by

$(c)'$ $(g_{\alpha-1} * S_\alpha *f)(t) \in D(A)$ for $t \in
[0,\tau)$ and $A(g_{\alpha-1} * S_\alpha *f)(t)$ is continuous on
$[0,\tau)$.
\end{prop}
\begin{proof}
The equivalence of $(a)$, $(b)$ and $(c)'$  was given in \cite{LL}
for the case $ \alpha \in [1,2]$. The case $\alpha \in (0,1)$ can be
proved similarly.
\end{proof}

As a corollary we have
\begin{cor}\label{inho-cor}
Let $\alpha \in (0,2]$. Suppose that $A$ is the generator of an
$\alpha$-times resolvent family. Then (\ref{ACPf0}) has a strong
solution on $[0,\tau)$ if one of the following conditions is
satisfied:

(a) $f $ is continuously differentiable on $[0,\tau)$.

(b) $\alpha \in [1,2]$, $f(t) \in D(A)$ for $t \in [0,\tau)$ and
$Af(t) \in L^1_{loc}([0,\tau);X)$.

(c) $\alpha \in (0,1)$, $f(t) \in D(A)$ for $t \in [0,\tau)$ and
$g_\alpha*f $ is continuously differentiable on $[0,\tau)$.
\end{cor}

If $A$ generates an $\alpha$-times resolvent family $S_\alpha$, then
for $x \in D(A^n)$ by using (\ref{resolvent equation}) several times
we have
\begin{equation}\label{Salphan}
\begin{split}
S_{\alpha}(t)x &= x + (g_\alpha * S_{\alpha})(t)Ax \\
&=x + (g_\alpha * 1)(t)Ax +(g_\alpha* (g_\alpha
* S_{\alpha}))(t)A^2x \\
&= x + g_{\alpha+1}(t)Ax + (g_{2\alpha} *
S_{\alpha})(t) A^2 x \\
&= \cdots\\
&= x + g_{\alpha+1}(t) Ax +\cdots + g_{(n-1)\alpha+1}(t) A^{n-1}x +
(g_{n\alpha}*S_{\alpha})(t)A^nx,
\end{split}
\end{equation}
which leads to

\begin{lem}
If $A$ generates an $\alpha$-times resolvent family $S_\alpha$, then
for $x \in D(A^n)$ with $n\alpha \ge 1$, $S_\alpha(t)x$ is
differentiable and
$$
\frac{d}{dt}(S_\alpha(t)x) = \sum_{k=1}^{n-1} g_{k\alpha}(t)A^k x +
(g_{n\alpha -1} * S_\alpha)(t)A^n x, \quad t >0.
$$
\end{lem}

 In
particular, let $\alpha = 1/m$ with $m \in \mathbb{N}$, we obtain

\begin{prop}\label{prop1/m}
Let $m \in \mathbb{N}$. Suppose that $A$ generates a $(1/m)$-times
resolvent family $S_{1/m}$. Then for each $x\in D(A^{m})$,
$S_{1/m}(\cdot)x$ solves the fractional Cauchy problem
\begin{equation}\label{FACP2}
\begin{split}
&D_t^{1/m} u(t) = A u(t), \quad t > 0,\\
&u(0)=x,
\end{split}
\end{equation}
and the initial value problem
\begin{equation}\label{ACPf1}
\begin{split}
&v'(t) = A^{m} v(t) + \sum_{k=1}^{m-1}g_{{k}/{m}}(t)A^k x, \quad t > 0,\\
&v(0)=x.
\end{split}
\end{equation}
\end{prop}

\begin{rems}\label{theorem 1/m rem}
(a) If $A$ generates a $C_0$-semigroup, then by the subordination
principle $A$ generates an (analytic) $(1/m)$-times resolvent
family. So Proposition \ref{prop1/m} gives Theorem 3.3 in \cite{KL}
immediately.

(b) Note that $A^m$ does not necessarily generate a $C_0$-semigroup
when $A$ generates a $1/m$-resolvent family, we cannot obtain the
uniqueness of solution of (\ref{ACPf1}) without any further
assumption on the operator $A$ and a counterexample was given in
\cite{BMN}.
\end{rems}

For the corresponding inhomogeneous problems, we have
\begin{prop}\label{1/m}
Let $m\geq 2$ be fixed. Assume that $A$ is the generator of a
$(1/m)$-times resolvent family $S_{1/m}$, then for $x \in D(A^m)$,
$f(t) \in C(\Real_+, D (A^{m}))$, the function $ S_{{1}/{m}}(t) x +
(S_{{1}/{m}}*f)(t) $ solves the two equations:
\begin{equation}\label{inho-1}
\begin{array}{l}
D_t^{{1}/{m}} u(t) = A u(t) + (g_{(1-1/m)}*f)(t),\quad  t> 0 \\
u(0) = x\\
\end{array}
\end{equation}
and
\begin{equation}
\begin{array}{l}\label{inho-2}
v'(t) = A^m v(t) + \sum\limits_{k=1}^{m-1} g_{{k}/{m}}(t) A^k x +
\sum\limits_{k=0}^{m-1}
(g_{{k}/{m}}* A^kf)(t) , \quad t> 0 \\
v(0) = x.\\
\end{array}
\end{equation}
 \end{prop}
\begin{proof}
Since $g_{1/m}* (g_{(1-{1/ m})}*f) = g_1*f$ is differentiable and
$f(t) \in D(A)$ for all $t \ge 0$, by Proposition
\ref{strong-solution} and Corollary \ref{inho-cor} (c), $S_{1/
m}(t)x + (S_{1/ m}*f)(t)$ solves (\ref{inho-1}). It remains to show
that it is also a solution of (\ref{inho-2}). By Proposition
\ref{prop1/m}, we only need to show that $S_{1/ m}*f$ is
differentiable, $(S_{1/ m}*f)(t) \in D(A^m)$ and
$$
(S_{1/ m}*f)'(t) = A^m (S_{1/ m}*f)(t) + \sum\limits_{k=0}^{m-1}
(g_{{k}/{m}}* A^k f)(t), \quad t > 0.
$$
This follows from (\ref{Salphan}) since
\begin{eqnarray*}
&\mbox{}& (S_{1/m}*f)(t)\\
 &=& \int_0^t S_{1/m}(t-s)f(s) ds \\
&=& \int_0^t \Big[f(s) + \sum_{k=1}^{m-1} g_{{k\over
m}+1}(t-s)A^{k}f(s)+ (g_1*S_{1/m})(t-s)A^m f(s)\Big]ds.
\end{eqnarray*}
\end{proof}

Next we will discuss the connections between some pairs of the
Cauchy problems of fractional order (not necessarily a rational
number) and first order.

First, we have the following direct consequences of Corollary
\ref{bounded-cor}.

\begin{thm}\label{general relation homogeneous}
(a) Let $\alpha \in (0,1)$ and $-A \in \mathcal{C}_1(0)$. The
fractional Cauchy problem
\begin{equation}\label{FACP0}
\begin{split}
&D_t^\alpha v(t) = -A v(t), \quad t > 0,\\
&v(0)=x,
\end{split}
\end{equation}
is well-posed and its unique solution is given by
\begin{equation*}
\begin{split}
v(t) = \int_0^\infty \varphi_\alpha (t,s) u(s) ds, \quad t>0,
\end{split}
\end{equation*}
for each $x \in D(A)$, where $\varphi_\alpha$ is given as in
Corollary \ref{bounded-cor} and $u$ is the solution to the Cauchy
problem
\begin{equation}\label{ACP000}
\begin{split}
&u'(t) = -A u(t),\quad t > 0,\\
&u(0)=x.
\end{split}
\end{equation}

(b) Let $\alpha \in (0,1)$ and $-A \in \mathcal{C}_1(0)$. The
fractional Cauchy problem
\begin{equation}\label{ACP inserted}
\begin{split}
& v'(t) = -A^\alpha v(t), \quad t > 0,\\
& v(0)=x,
\end{split}
\end{equation}
is well-posed and its unique solution is given by
\begin{equation*}
\begin{split}
v(t) = \int_0^\infty p_\alpha (t,s) u(s) ds, \quad t>0,
\end{split}
\end{equation*}
for each $x \in D(A)$, where $p_\alpha$ is given as in Corollary
\ref{bounded-cor} and $u$ is the solution to the Cauchy problem
(\ref{ACP000}).

(c) Let $\alpha \in (0,1)$ and $-A \in \mathcal{C}_1(0)$. The
fractional Cauchy problem
\begin{equation}\label{FACP1}
\begin{split}
&D_t^\alpha v(t) = -A^{\alpha} v(t), \quad t > 0,\\
&v(0)=x,
\end{split}
\end{equation}
is well-posed and its unique solution is given by
\begin{equation*}
\begin{split}
v(t) = \int_0^\infty f_{1,\alpha}^\alpha (t,s) u(s) ds, \quad t>0,
\end{split}
\end{equation*}
for each $x \in D(A^{\alpha})$, where $f_{1,\alpha}^\alpha$ is given
as in Corollary \ref{bounded-cor} and $u$ is the solution to the
Cauchy problem (\ref{ACP000}).

(d) Let $\beta \in (1,2]$ and $-A \in \mathcal{C}_\beta(0)$. The
Cauchy problem (\ref{ACP000}) is well-posed and its unique solution
is given by
\begin{equation*}
\begin{split}
u(t) = \int_0^\infty \varphi_{1/\beta} (t,s) v(s) ds, \quad t> 0,
\end{split}
\end{equation*}
for each $x \in D(A)$, where $v$ is the solution to the fractional
Cauchy problem
\begin{equation}\label{insert 2}
\begin{split}
&D_t^\beta v(t) = -A v(t), \quad t > 0,\\
&v(0)=x,\, v'(0)=0.
\end{split}
\end{equation}

(e) Let $\beta \in (1,2]$ and $-A \in \mathcal{C}_\beta(0)$. The
Cauchy problem
\begin{equation}\label{ACP1}
\begin{split}
&u'(t) = -A^{1/\beta} u(t),\quad t > 0,\\
&u(0)=x,
\end{split}
\end{equation}
is well-posed and its unique solution is given by
\begin{equation*}
\begin{split}
u(t) = \int_0^\infty f_{\beta,1}^{1/\beta} (t,s) v(s) ds, \quad t>
0,
\end{split}
\end{equation*}
for each $x \in D(A^{1/\beta})$, where $f_{\beta,1}^{1/\beta}$ is
given as in Corollary \ref{bounded-cor} and $v$ is the solution to
the fractional Cauchy problem (\ref{insert 2}).
\end{thm}

\begin{rem}
(a) In Theorem \ref{general relation homogeneous} (a) and (c), if
$A$ generates an analytic $C_0$-semigroup, then the restriction on
$\alpha$ can be relaxed by using Proposition \ref{analytic-cor}.

(b) By using the generalized subordination principle in Theorem
\ref{main} and Proposition \ref{strong-solution} one can also
consider the  inhomogeneous fractional Cauchy problems.
\end{rem}

\begin{rem}\label{stochastic 1} The results in Theorem \ref{general relation
homogeneous} can be interpreted in terms of stochastic solutions.
Let $0 < \alpha < 1$ and $X$ be a Markov process with a semigroup
$T(t)f(x) = \mathbb{E}(f(X(t)))$ generated by $-A$ and let $E(t) :=
\inf\{x > 0: D(t) > t\}$ be the inverse or hitting time process of
the stable subordinator $D(t)$, independent of $X$, with
$\mathbb{E}(e^{-sD(t)}) = e^{-ts^\alpha}$. If $u$ is a solution to
the problem
\begin{equation}\label{aadje}
\begin{split}
u'(t) = -A u(t); \quad u(0) = f(x),
\end{split}
\end{equation}
then

(a) the problem
\begin{equation*}
\begin{split}
D_t^\alpha v(t) = -A v(t); \quad v(0) = f(x),
\end{split}
\end{equation*}
has a unique solution given by
\begin{equation*}
\begin{split}
v(t) = \mathbb{E}(f(X(E(t)))) = \int_0^\infty u(s) f_{E(t)}(s) ds,
\end{split}
\end{equation*}
where $f_{E(t)}(s)$ is the density of the inverse stable
subordinator of index $\alpha$ (see also Theorem 3.3 in
\cite{BWM2});

(b) the problem
\begin{equation*}
\begin{split}
W'(t) = -A^\alpha W(t); \quad W(0) = f(x),
\end{split}
\end{equation*}
has a unique solution given by
\begin{equation*}
\begin{split}
W(t) = \mathbb{E}(f(X(D(t)))) = \int_0^\infty u(s) f_{D(t)}(s) ds,
\end{split}
\end{equation*}
where $f_{D(t)}(s)$ is the density of the stable subordinator of
index $\alpha$;

(c) the problem
\begin{equation*}
\begin{split}
D_t^\alpha v(t) = -A^\alpha v(t); \quad v(0) = f(x),
\end{split}
\end{equation*}
has a unique solution given by
\begin{equation*}
\begin{split}
v(t) = & \mathbb{E}(f(X(D(E(t))))) = \int_0^\infty W(s) f_{E(t)}(s)
ds\\
= & \int_0^\infty \bigg( \int_0^\infty u(r) f_{D(s)} (r) dr \bigg)
f_{E(t)}(s) ds
\end{split}
\end{equation*}
with $W$ given in (b);

(d) if in addition that for some $\beta \in (1, 2]$ the fractional
Cauchy problem
$$
D^\beta V(t) = -A V(t);\quad  V(0)= f(x),\, V'(0)=0
$$
 is well-posed, then the solution of (\ref{aadje}), $u$, is subordinated to $V$ by
\begin{equation*}
u(t) =   \int_0^\infty V(s) h_{E(t)}(s)ds,
\end{equation*}
where $h_{E(t)}(s)$ is the density of the inverse stable
subordinator of index $1/\beta$;

(e) if the assumptions of (d) hold, then the solution to the Cauchy
problem
$$
v'(t) = -A^{1/\beta}v(t);\quad v(0)= f(x),
$$
is connected to $V$ by
$$
v(t) =\int_0^\infty u(s) h_{D(t)}(s) ds =\int_0^\infty \bigg(
\int_0^\infty V(r) h_{E(s)} (r) dr \bigg) h_{D(t)}(s) ds
$$
where $h_{E(t)}$ is as in (d) and $h_{D(t)}(s)$ is the density of
the stable subordinator of index $1/\beta$.
\end{rem}

We end this paper with two examples.
\begin{example}\label{example3}
Let $\rho>0$ and $m\in\mathbb{N}$. Consider the fractional
relaxation equation (cf. \cite{GM})
\begin{eqnarray}\label{f=0}
\begin{split}
&D_t^{1/m} u(t)=-\rho u(t),\quad t>0,\\
&u(0)=x.
\end{split}
\end{eqnarray}
The solution of (\ref{f=0}) is given by $u(t) = xE_{1/m}(-\rho
t^{1/m})$. By Proposition \ref{prop1/m}, $u(t)$ also solves
%Denote by $S_{1/m}$ the $\alpha$-times resolvent family generated by
%$\rho$. It is clear that $S_{1/m}(\cdot)x$ is the unique solution of
%(\ref{f=0}) for each $x\in X$. Moreover,
\begin{eqnarray}\label{1}
\begin{split}
&v'(t) = (-\rho)^m v(t) + \sum_{k=1}^{m-1}g_{\frac{k}{m}}(t) (-\rho)^k x,\quad t>0,\\
&v(0)=x.
\end{split}
\end{eqnarray}
Note that the solution of (\ref{1}) is unique. Therefore, the
problem (\ref{f=0}) is equivalent to the problem (\ref{1}).
\end{example}

\begin{example}\label{example4}

By Theorem \ref{general relation homogeneous}, the solution of the
fractional diffusion equation of order $0<\alpha \le 1$
\begin{equation}
\begin{split}
D_t^\alpha u(t,x) &= \Delta u(t,x),\quad t >0, \\
u(0,x)&= f_0(x)
\end{split}
\end{equation}
is given by $u(t,x) = \int_0^\infty \varphi_\alpha(t,s)
(T(s)f_0)(x)ds$, where $T$ is the Gaussian semigroup generated by
$\Delta$. Since
$$
T(s) f(x) = (k_s* f)(x)= (4\pi s)^{-n/2} \int_{\Real^n}
e^{-|x-y|^2/4s} f_0(y) dy,
$$
we have
$$
u(t,x) = \int_{\Real^n} \Big[\int_0^\infty \varphi_\alpha(t,s)(4\pi
s)^{-n/2}  e^{-|x-y|^2/4s }ds\Big] f_0(y) dy.
$$
See also \cite{SW}.

 %For $1\leq p<\infty$, let
%$\Delta_p$ be the Laplacian on $L^p(\mathbb{R}^n)$ with maximal
%domain:
%\begin{eqnarray*}
%\begin{split}
%&\Delta_pf:=\Delta f,\\
%&D(\Delta_p):=\{f\in L^p(\mathbb{R}^n): \Delta f\in
%L^p(\mathbb{R}^n)\},
%\end{split}
%\end{eqnarray*}
%where $\Delta f$ is defined in the distributional sense. It is
%well-known that $\Delta_p$ is the generator of the Gaussian
%semigroup $T_p$, a bounded analytic $C_0$-semigroup of angle
%$\pi/2$, so that $(-\Delta_p)^{1/m}$ generates a bounded analytic
%$(1/m)$-times resolvent family $S_{1/m}$. By Proposition \ref{1/m},
%for each $f \in D((-\Delta_p)^{1/m})$, $S_{1/m}(\cdot)x$ solves
%uniquely the problems
%\begin{eqnarray*}
%\begin{split}
%&D_t^{1/m} u(t,x) = -(-\Delta_p)^{1/m} u(t,x), \quad t>0,\\
%&u(0,x)=f(x).
%\end{split}
%\end{eqnarray*}
%and
%\begin{eqnarray*}
%\begin{split}
%&\frac{\partial}{\partial t}u(t,x) = (-1)^m \Delta_p u(t,x)+\sum_{k=1}^{m-1} g_{\frac{k}{m}}(t) (-\Delta_p)^{k/m} f(x), \quad t>0,\\
%&u(0,x)=f(x).
%\end{split}
%\end{eqnarray*}

\end{example}

%\begin{rem}
%These examples show that some fractional Cauchy problems can be
%transferred equivalently to some initial value problems to some
%extent, which furnish an approach to discuss the solutions of
%fractional evolution equations. In addition, many other examples in
%practice problems can be found in
%\cite{BM}-\cite{BSMW},\cite{KL},\cite{MBB}-\cite{MS2},\cite{
%SKDR}-\cite{SBMW} and \cite{Za}.
%\end{rem}

\vskip8pt

{\bf Acknowledgements}\hspace{0.25cm} The authors are grateful to
the referee for the valuable comments and suggestions, and
especially for the suggestion of giving the interpretation of our
results for stochastic solutions. Remark \ref{stochastic 1} is in
fact suggested by the referee.

\end{document}